\documentclass{amsart}
\usepackage{amsmath, amssymb, amsfonts, amscd, amsthm, mathrsfs}
\usepackage{tikz}
\usetikzlibrary{matrix,arrows}
\usepackage{comment}

\renewcommand{\hat}{\widehat}

\renewcommand{\hat}{\widehat}

\renewcommand{\bar}{\overline}

\newcommand{\bV}{\bar{V}}
\newcommand{\bff}{\bar{f}}
\newcommand{\bmu}{\pmb{\mu}}
\newcommand{\balpha}{\pmb{\alpha}}

\newcommand{\A}{\mathbf{A}}

\newcommand{\F}{\mathbf{F}}

\newcommand{\Q}{\mathbf{Q}}
\newcommand{\Z}{\mathbf{Z}}
\newcommand{\R}{\mathbf{R}}

\newcommand{\ba}{\mathbf{a}}
\newcommand{\bc}{\mathbf{c}}
\newcommand{\bv}{\mathbf{v}}
\newcommand{\bg}{\mathbf{g}}
\newcommand{\bx}{\mathbf{x}}
\newcommand{\bu}{\mathbf{u}}
\newcommand{\bz}{\mathbf{z}}

\newcommand{\bt}{\mathbf{t}}

\newcommand{\bbA}{\mathbb{A}}

\newcommand{\cD}{\mathscr{D}}
\newcommand{\cG}{\mathscr{G}}
\newcommand{\cH}{\mathscr{H}}
\newcommand{\cK}{\mathscr{K}}

\newcommand{\cO}{\mathscr{O}}
\newcommand{\cP}{\mathscr{P}}
\newcommand{\cQ}{\mathscr{Q}}
\newcommand{\cU}{\mathscr{U}}
\newcommand{\cZ}{\mathscr{Z}}

\newcommand{\ha}{\hat{a}}

\newcommand{\Tr}{\mathrm{Tr}}
\newcommand{\ord}{\mathrm{ord}}

\newcommand{\Spec}{\mathrm{Spec}}

\newcommand{\Mat}{\mathrm{Mat}}

\theoremstyle{plain}
\newtheorem{theorem}{Theorem}[section]
\newtheorem{proposition}[theorem]{Proposition}
\newtheorem{lemma}[theorem]{Lemma}

\theoremstyle{definition}
\newtheorem{definition}[theorem]{Definition}
\newtheorem{remark}[theorem]{Remark}
\newtheorem{examples}[theorem]{Examples}

\newtheorem{acknowledgments}{Acknowledgments}

\begin{document}

\title[On a theorem of Ax and Katz]
{On a theorem of Ax and Katz}

\author{Hui June Zhu}
\address{
Department of mathematics,
SUNY at Buffalo,
Buffalo, NY 14260}
\email{hjzhu@math.buffalo.edu}
\date{\today}
\maketitle

\begin{abstract}
The well-known theorem of Ax and Katz gives a $p$-divisibility bound
for the number of rational points on an algebraic variety $\bV$ over a finite field
of characteristic $p$ in terms of the degree and number of variables of 
defining polynomials of $\bV$. It was strengthened by Adolphson-Sperber 
in terms of Newton polytope of the support set $\cG$ of $\bV$. 
In this paper we prove that for every generic algebraic variety $V$ over $\bar\Q$ 
supported on $\cG$ the Adolphson-Sperber bound can be achieved on special fibre
at $p$ for a set of prime $p$ of positive density in $\Spec(\Z)$.
Moreover we show that if $\cG$ has certain combinatorial conditional number
nonzero then the above bound is achieved at special fibre at $p$ for all large enough
$p$. 
\end{abstract}

\section{Introduction}
\label{S:introduction}

In this paper $p$ is a prime number and $q=p^a$ for some integer $a>0$.
Let $\bV$ be an algebraic variety over $\F_q$ defined by a set of non-constant polynomials 
$\bff_1,\ldots,\bff_r$ in $\F_q[x_1,\ldots, x_n]$ in $n$ variables. 
We study $p$-divisibility of 
the cardinality $|\bV(\F_q)|$ of the set of 
$\F_q$-rational points on $V$. 
This problem (in a slightly different form) was first proposed by Artin (see \cite{Art65}) in 1935,
a first bound was given by Chevalley (see \cite{Che36}) and Warning (see \cite{War36})
using elementary method. 
From then on $p$-divisibility problem is also known as
{\em Chevalley-Warning problem}.
Ax (see \cite{Ax64}) 
and subsequently Katz (see \cite{Kat71})
used Dwork's method to give the following bound which is well known as the
{\em Ax-Katz bound},
\begin{equation}\label{E:Ax-Katz}
\ord_q|\bV(\F_q)| \geq \left\lceil\frac{n-\sum_{j=1}^{r} \deg(\bff_j)}
{\max_{1\leq j\leq r}\deg(\bff_j)} \right\rceil.
\end{equation}
(See also \cite{Wan89} for an elementary proof.)
For each integral point $\bg=(g_1,\ldots,g_n)\in\Z_{\geq 0}^n$
write $\sigma_p(\bg):=\sum_{i=1}^{n}\sigma_p(g_i)$ where $\sigma(g_i)$ denote the sum of 
$p$-adic digits in $g_i\in\Z_{\geq 0}$. 
Define $\sigma_p(\bff_j):=\max_{\bg\in \cG_j}(\sigma_p(\bg))$. 
Moreno-Moreno observed that 
one can always reduce $\bV/\F_q$ to $\bV'/\F_p$ where $\bV'$ is defined by
a set of $ra$ polynomials in $na$ variables with degrees $\leq \sigma_p(\bff_j)$ of $\bff_j$. 
They apply Ax-Katz's bound on $\bar{V'}$ and get the {\em Moreno-Moreno bound}:
\begin{equation}\label{E:MM}
\ord_q |\bV(\F_q)| \geq  \frac{1}{a}\left \lceil a\cdot \frac{n-\sum_{j=1}^{r} \sigma_p(\bff_j)}
{\max_{1\leq j\leq r}\sigma_p(\bff_j)}\right\rceil.
\end{equation}
This bound only potentially improves Ax-Katz bound for small $p$, namely for $p<\max_j\deg(\bff_j)$.
Let $\bff:=z_1\bff_1+\ldots+z_r\bff_r$ where $z_1,\ldots,z_r$ are new variables
and $\Delta(\bff)$ its Newton polytope in $\R^{n+r}$.
Let $w(\bff)$ be the least positive rational number $c$
such that the dilation $c\cdot \Delta(\bff)$ contains an integral point of all positive coordinates.
Adolphson-Sperber proved in \cite{AS87}
the following {\em Adolphson-Sperber bound} that strengthens Ax-Katz
\begin{equation}\label{E:AS}
\ord_q |\bV(\F_q)|\geq w(\bff)-r.
\end{equation}
We prove in this paper that Adolphson-Sperber bound is 
an asymptotic generic bound in Theorem \ref{T:main} below.
We write $\bmu=w(\bff)-r$ for the rest of the paper.

For $1\leq j\le r$ let $f_j=\sum_{\bg\in \cG_j}a_{j,\bg}\bx^{\bg}$ 
where $a_{j,\bg}\neq 0$.
Let $V:=V(f_1,\ldots,f_r)$ 
be the algebraic variety defined by the vanishing of $f_1,\ldots,f_r$.
Write $\bbA^{\cG_1,\ldots,\cG_r}$ for the space of all such algebraic varieties 
$V$, and write $\bbA^{\cG_1,\ldots,\cG_r}(K)$ for all such $V$ defined over $K$,
that is, $f_j\in K[x_1,\ldots,x_n]$.

\begin{theorem}\label{T:main}
Fix $\cG_1,\ldots,\cG_r$ in $\Z_{\ge 0}^{n}$, and let $\bbA^{\cG_1,\ldots,\cG_r}$ be
the space of all algebraic varieties supported on $\cG_1,\ldots,\cG_r$.
\begin{enumerate}
\item
For every generic $V$ in $\bbA^{\cG_1,\ldots,\cG_r}(\bar\Q)$
there exists a set of prime numbers $p$ of positive density in $\Spec(\Z)$ such that
\begin{eqnarray*}
\ord_q |\bar{V}(\F_q)|  &=& \bmu.
\end{eqnarray*}
\item 
There exists a conditional number function
$\bc(\cG_1,\ldots,\cG_r)$ defined explicitly in (\ref{E:conditional}), such that 
if $\bc(\cG_1,\ldots,\cG_r)\neq 0$, then 
for every $V$ in 
$\bbA^{\cG_1,\ldots,\cG_r}(\bar\Q)$
for every prime $p$ large enough 
\begin{eqnarray*}
\ord_q |\bar{V}(\F_q)|  &=& \bmu.
\end{eqnarray*}
\end{enumerate}
\end{theorem}

This theorem shows that the Adolphson-Sperber $p$-divisibility bound
$\bmu$ is {\em asymptotically generically} sharp in the sense it can 
be achieved for a generic $V$ over $\bar\Q$ at infinitely many special fibers.
But only when the support set $(\cG_1,\ldots,\cG_r)$ satisfies certain condition
can a generic $V$ over $\Q$ 
achieve the $p$-divisibility bound for all but finitely many $p$.

\begin{examples}\label{Ex:1}
Suppose 
$V(f)$ is any hypersurface defined by 
$f=\sum_{\bg\in\cG} a_\bg \bx^\bg$
in $\Q[x_1,\ldots,x_n]$ 
where $a_\bg\in\Q^*$ and $\cG$ consists of all 
$\bg=(g_1,\ldots,g_n)\in\Z_{\geq 0}^n$ with $|\bg|:=\sum_{i=1}^{n} g_i=d$.
Let $n\geq d$. Then Theorem \ref{T:main} implies that
$
\ord_p|\bar{V}(\F_p)|\geq  \lceil{\frac{n-d}{d}}\rceil,
$
the same as the Ax-Katz bound.
A consequence of Theorem \ref{T:main} (1) is that  
for all generic such $V(f)$ over $\Q$ 
we have 
$
\ord_p|\bar{V}(\F_p)|=\lceil \frac{n-d}{d}\rceil
$
for primes $p$ in a set of positive density in $\Spec(\Z)$.
In comparison, Katz showed in 
\cite[Section 5]{Kat71}  there is an explicit algebraic surface
over $\F_p$ for every prime $p$ that his bound is achieved. 
\end{examples}

\begin{examples}\label{Ex:2}
Suppose $V(f)$ is any hypersurface with 
$f=a_1 x_1^3 x_2^3+ a_2 x_2^2 x_3^2$ where $a_1,a_2\in\Z-\{0\}$.
One can check that $|\bar{V}(\F_p)|=p(2p-1)$ for all $p$, hence 
$\ord_p|\bar{V}(\F_p)|=1$.
By Theorem \ref{T:main}(2) 
we have $\ord_p(|\bar{V}(\F_p)|)\ge 1$ 
and the equality holds for all $V(f)$ and at all prime $p$ large enough.
In comparison, Ax-Katz bound says that 
$\ord_p(|\bar{V}(\F_p)|)\geq 0$ which is weaker in this example.
\end{examples}

Our proof uses Dwork method. We first briefly recall necessary $p$-adic theory to study the 
number of rational points $|\bar{V}(\F_q)|$ of algebraic variety $\bar{V}$ over $\F_q$
in Section \ref{S:Dwork}, then we prepare some $\A$-polynomials where variables $\A$ parametrize
the coefficients of defining polynomials of $V$ over $\bar\Q$ in 
Section \ref{S:Hasse}. This section is technical and combinatorial.   
We start to prove our theorem for algebraic variety $V$ over $\Q$ and then 
reduce the general case $V$ over $\bar\Q$ to that over $\Q$ immediately.
Our proof of the main theorem \ref{T:main} lies in Section \ref{S:proof}. 

\begin{acknowledgments}
We thank Regis Blache and Kiran Kedlaya for very helpful comments on earlier versions of 
this paper. 
\end{acknowledgments}

\section{Rational points and the trace}
\label{S:Dwork}

For the rest of the paper we fix nonempty subsets 
$\cG_1,\ldots,\cG_r$ in $\Z_{\geq 0}^n$. 
They may or may not be distinct.
Let $\bar{f}_1,\ldots,\bar{f}_r$ be any polynomials in $\F_q[x_1,\ldots,x_n]$ with 
supporting coefficient sets 
$\cG_1,\ldots,\cG_r$ respectively.
That is, for each $j=1,\ldots,r$ one can write
$\bar{f}_j=\sum_{\bg\in \cG_j}\bar{a}_{j,\bg}\bx^\bg$ for $\bar{a}_{j,\bg}\in\F_q^*$.
Let $\bar{f}=z_1\bar{f}_1+\ldots+z_r\bar{f}_r\in \F_q[x_1,\ldots,x_n,z_1,\ldots,z_r]$.
Let $C$ be a nonempty subset of $\{1,\ldots,n\}$
and we write $\Z_{>0}^C$ (resp. $\Z_{\geq 0}^C$) for the subset of $\Z_{\ge 0}^n$ 
with $i$-th component in $\Z_{\geq 1}$ (resp. $\Z_{\geq 0}$) 
if $i\in C$ and equal to $0$ if $i\not\in C$.
Let $B$ be a nonempty subset of $\{1,\ldots,r\}$, and let 
$\Z^B_{>0}$ be defined similarly.
Let  $\cG_{j,C}:=\cG_j\cap\Z_{\geq 0}^C$.
Write $\bv=(v_i)_{i=1}^{n}$ and $\bt=(t_j)_{j=1}^{r}$.
Let 
\begin{eqnarray}\label{E:Z}
\cZ_{B,C} &:=&\bigg\{ (\bt,\bv)\in \Z^B_{>0}\times\Z^C_{>0} \bigg|
\bv=\sum_{j\in B}\sum_{\bg\in \cG_{j,C}} u_\bg \bg, 
u_\bg\in\Q_{\geq 0},\\
&& t_j=\sum_{\bg\in \cG_{j,C}}u_\bg\in \Z_{\ge 1} \mbox{ for each $j\in B$}
\bigg\}.\nonumber
\end{eqnarray}
For every pair $(\bt,\bv)$ in $\cZ_{B,C}$
we write $\bz^{\bt}\bx^\bv=\prod_{j\in B}z_j^{t_j}\prod_{i\in C}x_i^{v_i}$.
Let 
\begin{eqnarray}\label{E:Q}
\cQ_{B,C} &:=&
\bigg\{
(\bu,\bv)\in
\Q_{\geq 0}^{\sum_{j\in B}|\cG_{j,C}|}\times\Z_{> 0}^{C}\bigg|\bv=\sum_{\bg}u_\bg \bg, u_\bg\in\Q_{\geq 0},
\nonumber\\
&&\bu=(u_\bg)_{\bg\in \cup_{j\in B} \cG_{j,C}}, 
\sum_{\bg\in\cG_{j,C}}u_\bg \in\Z_{\ge 1} \mbox{ for each $j\in B$} 
\bigg\}.
\end{eqnarray}
Write the subset of integral points in $\cQ_{B,C}$ by $(\cQ_{B,C})_\Z$.
Then we have the natural surjective map 
\begin{equation}\label{E:iota}
\iota: \cQ_{B,C}\twoheadrightarrow \cZ_{B,C}
\end{equation}
that sends all $(\bu,\bv)$ to $(\bt,\bv)$
with $t_j=\sum_{\bg\in\cG_{j,C}} u_\bg$.
Let $|\bt|:=\sum_{j=1}^{r}t_j$ and $|\bu|:=\sum_{j\in B}\sum_{\bg\in \cG_{j,C}}u_\bg$.
For $(\bu,\bv)$ in $\cQ_{B,C}$ and $(\bt,\bv)=\iota(\bu,\bv)$ in $\cZ_{B,C}$, we have
$|\bt|=|\bu|$.

Let $E_p(x)$ be the $p$-adic Artin-Hasse exponential function. 
We write $E_p(x)=\sum_{i=0}^{\infty}\delta_i x^i$ where
$\delta_i\in \Z_p\cap \Q$.
For $0\leq i\leq p-1$,
$
\delta_i =\frac{1}{i!}.
$
Let $\gamma$ be a root of $\log E_p(x)=\sum_{i=0}^{\infty}\frac{x^{p^i}}{p^i}$
with $\ord_p\gamma =\frac{1}{p-1}$ in $\bar\Q_p$.
Notice that 
\begin{eqnarray}\label{E:gamma-p}
\frac{\gamma^{p-1}}{p}\equiv -1 \bmod \gamma.
\end{eqnarray}

For each $(B,C)$ we define a $\Z_q[\gamma]$-algebra
\begin{equation}\label{E:H}
\cH_{B,C}:=\left\{\sum_{(\bt,\bv)\in \cZ_{B,C}}c_{\bt,\bv} \gamma^{|\bt|} \bz^{\bt} \bx^\bv
\bigg| c_{\bt,\bv}\in \Z_q \right\}.
\end{equation}
This is a subalgebra of $\Z_q[x_1,\ldots,x_n,\gamma z_1,\ldots,\gamma z_r]$.
Let $\bar{f}_{B,C}$ be the restriction of $\bar{f}$ for $j\in B$ and $i\in C$, namely,
$$\bar{f}_{B,C}=\sum_{j\in B}\sum_{\bg\in\cG_{j,C}}\bar{a}_{j,\bg} z_j\bx^\bg.$$ 
Let $\ha_{j,\bg}$ be the Teichm\"uller lifting of $\bar{a}_{j,\bg}$ to $\Z_q^*$,
then Dwork's splitting function of $\bar{f}_{B,C}$ is 
$
G_{\bar{f}_{B,C}}:=\prod_{j\in B}\prod_{\bg \in \cG_{j,C}} E_p(\gamma \ha_{j,\bg} z_j\bx^\bg)
$,
which lies in $\cH_{B,C}$. 
Write $\ba^{\bu}=\prod_{j\in B}\prod_{\bg\in \cG_{j,C}}a_{j,\bg}^{u_\bg}$
and write ${\hat\ba}^{\bu}$ for the corresponding Teichm\"uller lifting similarly.
Then its expansion is
\begin{eqnarray*}
G_{\bar{f}_{B,C}} &=& \sum_{(\bt,\bv)\in\cZ_{B,C}} G_{\bt,\bv}\gamma^{|\bt|} \bz^\bt \bx^\bv
\end{eqnarray*}
where its coefficients are given by 
\begin{eqnarray}
\label{E:G}
G_{\bt,\bv} &=& 
\sum_{(\bu,\bv)} (\prod_{\bu} \delta_{u_\bg}) \hat{\ba}^{\bu}
\end{eqnarray}
and the sum ranges over all 
$(\bu,\bv)\in (\iota^{-1}(\bt,\bv))\cap (\cQ_{B,C})_\Z$. 
Notice that $G_{\bt,\bv}=0$ if and only if 
$(\iota^{-1}(\bt,\bv))\cap (\cQ_{B,C})_\Z=\emptyset$. 

Define an operator $\psi_p$ on $\cH_{B,C}$ by
\begin{equation}\label{E:psi}
\psi_p\left(\sum_{(\bt,\bv)\in\cZ_{B,C}}c_{\bt,\bv} \gamma^{|\bt|} \bz^\bt \bx^\bv\right) 
= \sum_{(\bt,\bv)\in\cZ_{B,C}}c_{p\bt,p\bv} \gamma^{p|\bt|} \bz^\bt \bx^\bv.
\end{equation}
One can check that $\cH_{B,C}$ is closed under $\psi_p$.
Indeed, suppose 
$(\bt,\bv)\in\cZ_{B,C}$ and $(\bt,\bv)=(p\bt',p\bv')$ with $(\bt',\bv')$ with all corresponding entries of $\bt'$ and $\bv'$ in $\Z_{\geq 1}$, then $p\bv'=\bv=\sum u_\bg \bg$ and 
$pt_j'=t_j=\sum_{\bg\in \cG_{C,j}} u_\bg\in\Z_{\geq 1}$ with $u_g\in\Q_{\geq 0}$.
We have
$\bv'=\sum (u_\bg/p) \bg$ and 
$t_j'=\sum_{\bg\in\cG_{C,j}}(u_\bg/p)\in\Z_{\geq 1}$, so 
$(\bt',\bv')\in\cZ_{B,C}$ and hence $\cH_{B,C}$ is closed under the operator $\psi_p$.

Let $\tau$ be the Frobenius automorphism of $\Q_q(\gamma)$ over $\Q_p(\gamma)$
and its induced map on $\cH_{B,C}$ is
$\tau^{-1}(\sum c_{\bt,\bv}\gamma^{|\bt|}\bz^\bt \bx^\bv)=\sum \tau^{-1}(c_{\bt,\bv}\gamma^{|\bt|})\bz^\bt\bx^\bv$.
Let $\balpha_{\bar{f}_{B,C}}$ be the Dwork operator on $\cH_{B,C}$ defined by 
\begin{equation}
\balpha_{\bar{f}_{B,C}}:=\tau^{-1}\circ\psi_p\circ G_{\bar{f}_{B,C}}.
\end{equation}

Let $\cH^{(\ell)}_{B,C}$ denote the sub-algebra of 
$\cH_{B,C}$ with $|\bt|=\ell$, we have a decomposition
$\cH_{B,C}=\oplus_{\ell\in\Z_{\geq 0}}\cH^{(\ell)}_{B,C}$.
Note that 
$\psi_p(\sum c_{\bt,\bv}\gamma^\ell \bz^\bt\bx^\bv)
=\gamma^{(p-1)\ell}\sum c_{p\bt,p\bv}\gamma^\ell \bz^\bt\bx^\bv$.
Then $\psi_p(\cH_{B,C}) \subseteq \oplus_{\ell\in\Z_{\geq 0}}p^\ell \cH^{(\ell)}_{B,C}.$
Since $G_{\bar{f}_{B,C}}$ lies in $\cH_{B,C}$,
we have $G_{\bar{f}_{B,C}} \cdot \cH_{B,C}\subseteq\cH_{B,C}$. 
Then it follows
\begin{eqnarray*}
\balpha_{\bar{f}_{B,C}}(\cH_{B,C})\subseteq \oplus_{\ell\in\Z_{\geq 0}}p^\ell \cH^{(\ell)}_{B,C}.
\end{eqnarray*}
From now on we order elements $(\bt,\bv)$ in $\cZ_{B,C}$ so that $|\bt|$ is nondecreasing, and 
in this way $\cZ_{B,C}$ becomes a partially ordered set.
Choose the weighted monomial basis $\{\gamma^{|\bt|}\bz^\bt\bx^\bv\}$ for $\cH_{B,C}$ over $\Z_q$
with the pairs $(\bt,\bv)$ ranging in (the partially ordered set) $\cZ_{B,C}$.
Then 
\begin{equation}\label{E:DworkOperator}
\balpha_{\bar{f}_{B,C}}(\gamma^{|\bt'|}\bz^{\bt'} \bx^{\bv'})
=\sum_{(\bt,\bv)\in\cZ_{B,C}}(\tau^{-1}\gamma^{(p-1)|\bt|}G_{p\bt-\bt',p\bv-\bv'})\gamma^{|\bt|}\bz^{\bt}\bx^{\bv},
\end{equation}
where $G_{-,-}$ is as defined in (\ref{E:G}).
Write the infinite matrix 
\begin{eqnarray}\label{E:M}
M_{B,C}&:=& \left(\gamma^{(p-1)|\bt|}
G_{p\bt-\bt', p\bv-\bv'}\right)_{(\bt',\bv'), (\bt,\bv)}
\end{eqnarray}
where the row and column are indexed by the pairs $(\bt',\bv')$ and $(\bt,\bv)$ in $\cZ_{B,C}$, respectively. 
This matrix lies over $\Q_q[\gamma]$.
Then the matrix of $\balpha_{\bar{f}_{B,C}}$ with respect to the above weighted monomial basis 
is 
$\Mat(\balpha_{\bar{f}_{B,C}})=\tau^{-1} M_{B,C}$.
For any matrix $M$ we write 
\begin{equation}\label{E:Ma}
M^{[a]}:=M^{\tau^{a-1}}\cdots M^{\tau} M.
\end{equation}
For $a$ compositions of the Dwork operator
$\balpha_{\bar{f}_{B,C}}^a=\balpha_{\bar{f}_{B,C}}\circ \cdots \circ \balpha_{\bar{f}_{B,C}}$,
we have
$$
\Mat(\balpha_{\bar{f}_{B,C}}^a) = M_{B,C}^{[a]}.
$$

\begin{theorem}
\label{T:Dwork}
Let $\bar{V}$ be algebraic variety defined
by the polynomials $\bar{f}_1,\ldots, \bar{f}_r\in \F_q[x_1,\ldots,x_n]$ with $q=p^a$.
For any nonempty subsets $B,C$ in $\{1,\ldots,r\}$ and $\{1,\ldots,n\}$ respectively let
$M_{B,C}$ be the nuclear matrix defined in (\ref{E:M}).
Then 
\begin{eqnarray}\label{E:dwork2}
|\bar{V}(\F_q)| &=& q^n+\sum_{B,C}(q-1)^{|B|+|C|} q^{n-|B|-|C|}\Tr(M_{B,C}^{[a]}).
\end{eqnarray}
\end{theorem}
\begin{proof}
We have already shown above that 
$\balpha_{\bar{f}_{B,C}}$ is a nuclear operator on $\cH_{B,C}$
(see \cite{Dwo60} or \cite{Ser62}).
Our statement follows the same standard counting argument as that for 
 (3.5.4) in \cite{Kat71}.
\end{proof}

For a nuclear matrix $M$ over a $p$-adic valuation ring $R$,
we write $\ord_pM$ for the minimum $p$-adic order of all entries of $M$.

\begin{lemma}\label{L:dominate}
Let $q=p^a$ and $k\in \Z_{\ge 0}$.
Let $M$ be a nuclear matrix over a $p$-adic valuation ring 
of the block form
$$\frac{M}{p^k}=
\left(
\begin{matrix}
M_{11}  & M_{12}\\
p^{>0}M_{21} & p^{>0}M_{22}
\end{matrix}
\right)
 + (p^{>0})
$$
where $M_{11}$ is a square submatrix, $M_{ij}$ are all submatrix such that
$\ord_p M_{ij}\geq 0$.
Let $M^{[a]}$ be as defined in (\ref{E:Ma}).
Then we have
$
\frac{\Tr(M^{[a]})}{q^{k}} \equiv \Tr(M_{11}^{[a]})\bmod (p^{>0}).
$
In particular, if $\ord_q\Tr(M_{11}^{[a]})=0$ then 
$\ord_q \Tr(M^{[a]})=k$.
\end{lemma}

\begin{proof}
Notice that 
\begin{eqnarray*}
\frac{M^{[a]}}{p^{ak}}
&=& \frac{M^{\tau^{a-1}}}{p^k}\cdots \frac{M^\tau}{p^k} \frac{M}{p^k} 
\equiv 
\begin{pmatrix}
M_{11}^{[a]} & \star\\
0 & 0
\end{pmatrix}
\bmod (p^{>0})
\end{eqnarray*}
where $\ord_p\star\geq 0$.
Thus 
\begin{eqnarray*}
\frac{M^{[a]}}{q^{k}}\equiv 
\begin{pmatrix}
M_{11}^{[a]} & \star\\
0 & 0
\end{pmatrix}
\bmod (p^{>0}).
\end{eqnarray*}
Taking trace on both sides, we get
$
\frac{\Tr(M^{[a]})}{q^{k}} \equiv \Tr(M_{11}^{[a]})\bmod (p^{>0})
$
which proves our statement. 
\end{proof}

As the nuclear matrix $M_{B,C}$ has 
its entries as polynomials in coefficients $\bar\ba=(\bar{a}_{j,\bg})$ 
of  the defining polynomials $\bar{f}_j=\sum_{\bg\in\cG_j}\bar{a}_{j,\bg} \bx^\bg$,
we shall {\em deform} each entry to polynomials in variables $\A:=(A_{j,\bg})$.
Subsequently, the trace of $M_{B,C}$ is also {\em deformed} to 
a polynomial in $\A$. 
This idea is pronounced in the following Section \ref{S:Hasse}.

\section{$\A$-deformations and $\A$-polynomials}
\label{S:Hasse}

The sets $\cG_1,\ldots,\cG_r$ in $\Z_{\geq 0}^n$ are fixed.
Recall that $B,C$ are nonempty subsets in $\{1,\ldots,r\}$ 
and $\{1,\ldots,n\}$, respectively. 
Define an integral weight of each pair $(B,C)$ by 
\begin{eqnarray}\label{E:integral}
w_\Z(B,C)
:=
\min_{\bv\in\Z_{>0}^n}
\bigg\{
\sum_{j\in B}\sum_{\bg\in \cG_{j,C}}u_\bg
\bigg| 
\bv=\sum_{j\in B}\sum_{\bg\in \cG_{j,C}}u_\bg \bg, \mbox{ with }\\
\quad u_\bg\in\Q_{\geq 0}, 
\sum_{\bg\in\cG_{j,C}}u_\bg \in\Z_{\ge 1} \mbox{ for every $j\in B$}
\bigg\}.
\nonumber
\end{eqnarray}
If no such representation of $\bv$ described in (\ref{E:integral})
exists then assign $w_\Z(B,C)=\infty$.
Otherwise $|B|\leq w_\Z(B,C)\leq r$ where $|B|$ denotes the cardinality of $B$.
From the definition $\cZ_{B,C}$ from (\ref{E:Z}) 
one observes clearly
$w_\Z(B,C)=\min\{|\bt| \big| (\bt,\bv)\in\cZ_{B,C}\}$. 
Let $\cZ^{min}_{B,C}$ be the (finite) subset of $\cZ_{B,C}$ 
consisting of all $(\bt,\bv)$ in 
that the minimal bound $w_\Z(B,C)$ is realized, that is $|\bt|=w_\Z(B,C)$.
We can show by combinatorics that 
\begin{eqnarray}\label{E:b}
\bmu &=&\min_{B,C}\bigg\{n-|B|-|C|+w_{\Z}(B,C)\bigg\}.
\end{eqnarray}
Let $\cK$ be the (nonempty) set of all $(B,C)$ with $n-|B|-|C|+w_\Z(B,C)=\bmu$. 
By this definition one observes that $\cZ^{min}_{B,C}\neq \emptyset$  for each $(B,C)\in \cK$.
Hence 
\begin{eqnarray}\label{E:Zmin}
\cZ^{min} &:=&\bigcup_{(B,C)\in\cK} \cZ^{min}_{B,C}
\end{eqnarray}
is a nonempty
set consisting of all $(\bt,\bv)\in \Z_{>0}^B\times \Z_{>0}^C$ with
$|\bt|= \bmu-n+|B|+|C|=w_\Z(B,C)$.

Write $\A=(A_{j,\bg})_{j\in B,\bg\in \cG_{j,C}}$ for variables.
Let $G_{\bt,\bv}(\A)$ be the polynomial 
obtained via replacing each $\hat{\ba}$ in $G_{\bt,\bv}$ in (\ref{E:G})
by variables $\A$. 
\begin{equation}\label{E:polynomia-G}
G_{\bt,\bv}(\A):=\sum_{(\bu,\bv)} \prod_{\bu} (\prod_{\bg}\delta_{u_\bg})\A^\bu
\end{equation}
where $(\bu,\bv)\in (\iota^{-1}(\bt,\bv))\cap (\cQ_{B,C})_\Z$. 
Since all $\delta_i\in\Z_p\cap \Q$ for all $i$ we have
$G_{\bt,\bv}(\A)$ lies in $(\Z_p\cap \Q)[\A]$. 
For any $(\bt',\bv')$ and $(\bt,\bv)$ in $\cZ_{B,C}$ we have
\begin{equation}\label{E:G2}
G_{p\bt-\bt',p\bv-\bv'}(\A)=\sum_{(\bu'',p\bv-\bv')} \prod_{\bu''}(\prod_{\bg}\delta_{u''_\bg}) \A^{\bu''}
\end{equation}
where the sum is over all $(\bu'',p\bv-\bv')\in
(\iota^{-1}(p\bt-\bt',p\bv-\bv'))\cap (\cQ_{B,C})_\Z$. 
That is, $u''_\bg\in\Z_{\ge 0}$ and
$
\begin{array}{lll}
p\bv-\bv' &=& \sum_{\bg} u''_\bg \bg.
\end{array}
$
Let $M_{B,C}(\A)$ be the paramaterized nuclear matrix over $(\Z_p\cap\Q)[\gamma][\A]$ deforming
$M_{B,C}$ in (\ref{E:M})
\begin{equation}\label{E:MA}
M_{B,C}(\A):= \left(\gamma^{(p-1)|\bt|} G_{p\bt-\bt',p\bv-\bv'}(\A)\right)_{(\bt',\bv'),(\bt,\bv)},
\end{equation}
where the subindices range over $\cZ_{B,C}$.
Let
$$N_{B,C}(\A):=\left(G_{p\bt-\bt',p\bv-\bv'}(\A)\right)_{(\bt',\bv'),(\bt,\bv)}
$$
where $(\bt,\bv)$ and $(\bt',\bv')$ 
lie in $\cZ^{min}_{B,C}$. 
Then by (\ref{E:gamma-p}) and the definitions,
\begin{equation}\label{E:matrix}
(-1)^{w_\Z(B,C)} \frac{M_{B,C}(\A)}{p^{w_\Z(B,C)}}=
\begin{pmatrix}
N_{B,C}(\A) & M_{12}\\
p^{>0}M_{21} & p^{>0}M_{22}
\end{pmatrix}
+(p^{>0})
\end{equation}
for some submatrices $M_{12}, M_{21}, M_{22}$ all with $\ord_p(M_{ij})\geq 0$.

For any matrix $M(\A)$ over $K[\A]$ where $K$ is any field with $\tau$-action,
we define $\tau(\A):=\A^p$ and 
\begin{equation}\label{E:MaA}
M(\A)^{[a]}: = M(\A)^{\tau^{a-1}}\cdots M(\A)=M^{\tau^{a-1}}(\A^{p^{a-1}})\cdots M(\A).
\end{equation}

\begin{lemma}\label{L:dominate2}
Let $M_{B,C}(\A)^{[a]}$ and $N_{B,C}(\A)^{[a]}$
be defined as in (\ref{E:MaA}).
Then we have
$$
\frac{\Tr(M_{B,C}(\A)^{[a]})}{q^{w_\Z(B,C)}}\equiv 
(-1)^{a w_\Z(B,C)}\Tr(N_{B,C}(\A)^{[a]})\bmod p^{>0}.
$$
\end{lemma}
\begin{proof}
By (\ref{E:matrix}), 
we apply Lemma \ref{L:dominate}
to the matrix $M_{11}:=(-1)^{w_\Z(B,C)}N_{B,C}(\A)$
and $M:=M_{B,C}(\A)$.
\end{proof}

For any $p$-adic valuation ring $R$ and polynomial $F(\A)\in R[\A]$ 
let $\ord_p(F(\A))$ be the minimum of the $p$-adic orders
of all coefficients of $F(\A)$. 
Since $|\bt|\geq w_\Z(B,C)$ for all $(\bt,\bv)\in\cZ_{B,C}$, 
we have 
\begin{eqnarray}\label{E:trace-bound2}
\ord_q \Tr(M_{B,C}(\A)^{[a]})\geq \ord_q M_{B,C}(\A)^{[a]} \geq w_\Z(B,C).
\end{eqnarray}

\begin{definition}\label{D:representation}
Let $(\bt,\bv)\in\cZ_{B,C}$ where $\cZ_{B,C}$ is defined as in (\ref{E:Z}).
Consider the set of all {\em rational representations} of $(\bt,\bv)$ 
in terms of $\cG_1,\ldots,\cG_r$, namely,
$\bv=\frac{1}{d}\sum_{j\in B}\sum_{\bg\in\cG_{j,C}} r_\bg \bg$ 
with $d,r_\bg\in \Z_{\geq 0}$ and all $r_\bg$'s are coprime
such that  
$\frac{1}{d}\sum_{\bg\in \cG_{j,C}}r_\bg=t_j\in\Z_{\geq 1}$ for all $j\in B$ and 
$|\bt|=\sum_{j\in B} t_j=w_\Z(B,C)$. 
\end{definition}

Let $\cD$ be the  nonempty finite set of all denominators $d$ of rational representations 
of all $(\bt,\bv)\in\cZ^{min}$.

\begin{lemma}\label{L:key}
(1) For each $(\bt,\bv)\in \cZ_{B,C}\cap\iota((\cQ_{B,C})_\Z)$, 
there is 1-1 correspondence between 
a term with monomial $\prod_{\bg}A_\bg^{u_\bg}$ in the expansion of $G_{\bt,\bv}(\A)$ 
as in (\ref{E:polynomia-G})
and 
$(\bu,\bv)\in \iota^{-1}(\bt,\bv)\cap (\cQ_{B,C})_\Z$ such that $\bv=\sum_\bg u_\bg \bg$.
In particular $\sum_{\bg}u_\bg=|\bt|$.

(2) Let $(\bt,\bv)\in \cZ^{min}$. For each prime $p$,
there is a 1-1 correspondence between 
a nonzero term with monomial $\prod_{\bg}A_\bg^{u''_\bg}$ in 
the expansion of $G_{(p-1)\bt,(p-1)\bv}(\A)$
and
a rational representation of $(\bt,\bv)$
$$\bv=\frac{1}{d}\sum_{j\in B}\sum_{\bg\in\cG_{j,C}}r_\bg \bg$$
with all $d|(p-1)$,
in which $u''_\bg=\frac{(p-1)}{d}r_\bg\leq p-1$ and $\sum_{\bg} u''_\bg = (p-1)w_\Z(B,C).$

For $(\bt,\bv)\in\cZ^{min}$, we have
\begin{eqnarray*}
G_{(p-1)\bt,(p-1)\bv}(\A) 
&=&
\sum_{(\bu'',(p-1)\bv)} (\prod_{\bg}\delta_{u''_\bg})\A^{\bu''}\\\nonumber
&=&
\sum_{d\in\cD,d|(p-1)}\sum_{(\bu,\bv)}
\left(\prod_\bg \delta_{u''_\bg}\right)
\left(\prod_{\bg}A_{j,\bg}^{u''_\bg}\right)
\end{eqnarray*}
where the last sum ranges over all 
rational representations of $(\bt,\bv)$ with denominator $d|(p-1)$.
Each term has coefficient in $\Z_p^*$ and $A_{j,\bg}$-degree $\leq p-1$. 
\end{lemma}
\begin{proof}
The proof of Part (1) is elementary hence we omit it here.
Notice that since $|\bt|$ is minimal we have $\frac{r_\bg}{d} \leq 1$.
From Part (1) and the observation above  
we have $u''_\bg=(p-1)\frac{r_\bg}{d}\leq p-1$. This implies that $\delta_{u''_\bg} =\frac{1}{u''_\bg !}\in\Z_p^*$.
Then Part (2) follows.
\end{proof}

Define a Hasse polynomial in $(\Q\cap\Z_p)[\A]$ for the trace of $\balpha^a_{\bar{f}_{B,C}}$ as follows 
\begin{eqnarray}\label{E:Hpa}
H_p^{[a]}(\A)&:=&\sum_{(B,C)\in\cK}(-1)^{|B|+|C|+a w_\Z(B,C)} \Tr(N_{B,C}(\A)^{[a]}).
\end{eqnarray}
Write $H_p(\A):=H_p^{[1]}(\A)$ for simplicity then
\begin{equation*}
H_p(\A)=
\sum_{(B,C)\in\cK}(-1)^{|B|+|C|+w_\Z(B,C)}\sum_{(\bt,\bv)\in\cZ^{min}_{B,C}}
G_{(p-1)\bt,(p-1)\bv}(\A).
\end{equation*}

\begin{lemma}\label{L:Hp}
Let $p$ be any prime with $p\equiv 1\bmod d$ for all $d$ in $\cD$. 
Let $(\bt,\bv)\in\cZ^{min}$.

(1) 
Then 
the expansion of $G_{(p-1)\bt,(p-1)\bv}(\A)$ in Lemma \ref{L:key}(2)
has each term in $(\Z_p^*\cap\Q)[\A]$ with $A_{j,\bg}$-degree $\leq p-1$ and  
homogenous of total degree $(p-1) w_\Z(B,C)$.
Furthermore, $H_p(\A)\not\equiv 0\bmod p$. 

(2)
There is a constant $\theta(\cG_1,\ldots,\cG_r)$
such that for $p>\theta(\cG_1,\ldots,\cG_r)$ we have
$\Tr(N_{B,C}(\A)^{[a]})
\not\equiv 0\bmod p
$
and 
$H^{[a]}_p(\A)\not\equiv 0\bmod p$ is homogenous of total degree $(p^a-1)w_\Z(B,C)$,
\end{lemma}

\begin{proof}
(1) 
Notice that
$$H_p(\A)=\sum_{(B,C)\in\cK}
\sum_{(\bt,\bv)\in\cZ^{min}}(-1)^{|B|+|C|+w_\Z(B,C)}
G_{(p-1)\bt,(p-1)\bv}(\A).$$
The first statement is simply rephrasing the statement in 
Lemma \ref{L:key} (2).

By Lemma \ref{L:key}(2) and our hypothesis, 
each term of the expansion of $G_{(p-1)\bt,(p-1)\bv}(\A)$ 
is nonzero mod $p$.
Now we claim that for any $(\bt,\bv)\in \cZ^{min}_{B,C}$ and $(\bt',\bv')\in\cZ^{min}_{B',C'}$ 
where $(B,C)\neq (B',C')$ the corresponding $G_{(p-1)\bt,(p-1)\bv}$ and
$G_{(p-1)\bt',(p-1)\bv'}$ do not have common terms to cancel.
Indeed, suppose $B\neq B'$, then there exists 
$j\in B\backslash B'$ where 
$A_{j,\bg}^{\frac{p-1}{d}r_\bg}$ lies in a monomial of $G_{(p-1)\bt,(p-1)\bv}(\A)\backslash
G_{(p-1)\bt',(p-1)\bv'}(\A)$.
On the other hand suppose $C\neq C'$. 
Then there is $i\in C\backslash C'$ such that $v_i \neq 0$ while $v'_i=0$.
Take its corresponding rational representation 
$\bv=\sum_{\bg} \frac{r_\bg}{d}\bg$ has $r_\bg\neq 0$ for some $\bg\in \cup_{j\in B}\cG_{j,C}$
while $\bv'=\sum_{\bg}\frac{r'_\bg}{d'}\bg$ has $r'_\bg=0$ for $\bg\in \cup_{j\in B}\cG_{j,C}$.
This proves our claim. 
Therefore $H_p(\A)\not\equiv 0\bmod p$.

(2) 
To ease notation we omit subindex $(B,C)$ and $j$ for the rest of the proof.
Notice that 
\begin{eqnarray*}
\Tr(N(\A)^{[a]}) 
&=&
\sum G_{p\bt_1-\bt_2,p\bv_1-\bv_2}(\A^{p^{a-1}})\cdots G_{p\bt_a-\bt_1,p\bv_a-\bv_1}(\A)
\end{eqnarray*}
where $(\bt_i,\bv_i)$ ranges over all of $\cZ^{min}_{B,C}$.
By Lemma \ref{L:key}(2) and our hypothesis
each term in  $G_{p\bt-\bt,p\bv-\bv}(\A)$ lies in $(\Z_p^*\cap\Q)[\A]$
and $A_{j,\bg}$-degree $\leq p-1$. 
Write $\Tr(N(\A)^{[a]})=T_1(\A)+T_2(\A)$
where $T_1(\A)$ consists of all summands with all equal $(\bt_i,\bv_i)$ and 
$T_2(\A)$ with at least one not equal.
Each monomial in $T_1(\A)$ 
is of the form $\prod_{\bg}A_\bg^{\sum_{i=0}^{a-1}p^i u''_{\bg,i}}$
where $p\bv-\bv=\sum_{\bg}u''_{\bg,i}\bg$ and $0\leq u''_{\bg,i}\leq p-1$.
Coefficient of such monomial lies in $\Z_p^*$.
Notice that $T_1(\A)$ is homogenous of total degree
$(p^a-1)w_\Z(B,C)$, where each 
$A_\bg$-degree is $\sum_{i=0}^{a-1}p^i u''_{\bg,i}$.

On the other hand each monomial of 
$
T_2(\A)
$
is of the form
$\prod_{\bg}A_\bg^{\sum_{i=0}^{a-1}p^i u'_{\bg,i}}$;
Suppose this monomial form coincides with one lying in $T_1(\A)$
then $p\bv'-\bv''\equiv p\bv-\bv\bmod p$, and hence
$\bv''\equiv \bv\bmod p$. But $\bv'',\bv$ are bounded points by our hypothesis
so there is a constant 
$\theta(\cG_1,\ldots,\cG_r)$ such that 
for all $p>\theta(\cG_1,\ldots,\cG_r)$ we have $\bv''=\bv$ and hence $p\bv'-\bv''=p\bv'-\bv$.
By comparing $A_\bg$-degrees we have 
$
\sum_{i=0}^{a-1} p^i u''_{\bg,i} = \sum_{i=0}^{a-1} p^i u'_{\bg,i}
$
where $0\leq u''_{\bg,i}\leq p-1$ and $u'_{\bg,i}\in\Z_{\geq 0}$;
moreover, 
$
\sum_{i=0}^{a-1} u''_{\bg,i} = \sum_{i=0}^{a-1} u'_{\bg,i}.
$
This implies that $u''_{\bg,i}=u'_{\bg,i}$ for all $i$
and hence $p\bv'-\bv=p\bv-\bv$ and $\bv=\bv'$ constradicting our assumption above.
Therefore $T_2(\A)$ has no monomial terms in common with 
$T_1(\A)$. 

Using the same argument as that in Part (1)  we find that 
the summands in the expansion
of $H_p^{[a]}(\A)$ in (\ref{E:Hpa}) do not cancel out with each other for distinct $(B,C)\in\cK$,
and hence $H_p^{[a]}(\A)\not\equiv 0\bmod p$. 
\end{proof}

Below we shall single out a special case of $\cG_1,\ldots,\cG_r$ 
in which $H_p(\hat{a})\bmod p$ is a constant polynomial.
Let $\cD$ be the set of all denominators of rational representations of $(\bt,\bv)\in\cZ^{min}$.
Suppose $\cD=\{1\}$, that is,  for all $(\bt,\bv)\in\cZ^{min}$ 
\begin{eqnarray}\label{E:GG}
G_{(p-1)\bt,(p-1)\bv}(\A) 
&=&
\sum_{(\bu,\bv)\in\cZ^{min}_{B,C}}
\prod_{j\in B}\prod_{\bg\in\cG_{j,C}}(\frac{A_{j,\bg}^{p-1}}{(p-1)!})^{r_\bg}
\end{eqnarray}
where $r_\bg=0$ or $1$ by Lemma \ref{L:key}(2).
Let $m_{(\bt,\bv)}$ be the number of all such rational representations of $(\bt,\bv)\in\cZ^{min}_{B,C}$, and 
\begin{equation}\label{E:conditional}
\bc(\cG_1,\ldots,\cG_r):=\sum_{(B,C)\in\cK} (-1)^{w_\Z(B,C)}\sum_{(\bt,\bv)\in\cZ^{min}_{B,C}}m_{(\bt,\bv)}.
\end{equation}

One can compute and verify that Example \ref{Ex:2} satisfies the condition $\bc(\cG)\neq 0$.

\begin{proposition}\label{P:conditional}
Suppose $\cD=\{1\}$. Then 
$\bc(\cG_1,\ldots,\cG_r)\neq 0$
if and only if
for all $\ba\in \bbA^{\cG_1,\ldots,\cG_r}(\Q)$
and for $p$ large enough, we have
$H_p(\hat{\ba})\in\Z_p^*.$
\end{proposition}
\begin{proof}
Below we assume $\ba$ is integral by letting $p$ be large enough.
Combining with  Wilson's theorem  $(p-1)!\equiv -1\bmod p$ and
$\hat{\ba}^p\equiv \hat{\ba}\bmod p$, we have by (\ref{E:GG})
\begin{eqnarray*}
G_{(p-1)\bt,(p-1)\bv}(\hat\ba)
&=& \sum
\prod_{j\in B}\prod_{\bg\in\cG_{j,C}}(\frac{\hat{a}_{j,\bg}^{p-1}}
{(p-1)!})^{r_\bg}\\
&\equiv& \sum
\prod_{j\in B}\prod_{\bg\in\cG_{j,C}}(-1)^{r_\bg}\\
&\equiv & \sum
(-1)^{w_\Z(B,C)}\\
&\equiv &(-1)^{w_\Z(B,C)}m_{(\bt,\bv)}\bmod p
\end{eqnarray*}
where the sum is over all rational representations of $(\bu,\bv)$.
Then by (\ref{E:Hpa})
\begin{eqnarray*}
H_p(\hat{\ba})
&\equiv &\sum_{(B,C)\in\cK}(-1)^{|B|+|C|}
\sum_{(\bt,\bv)\in\cZ^{min}_{B,C}} m_{(\bt,\bv)}\\
&\equiv & (-1)^{n-\bmu}\sum_{(B,C)\in\cK} (-1)^{w_\Z(B,C)}
\sum_{(\bt,\bv)\in\cZ^{min}_{B,C}}m_{(\bt,\bv)}\\
&\equiv &(-1)^{n-\bmu} \bc(\cG_1,\ldots,\cG_r) \bmod p.
\end{eqnarray*}
As $\bc(\cG_1,\ldots,\cG_r)$ is a nonzero integer by our hypothesis, 
this proves that $\bc(\cG_1,\ldots,\cG_r)\not\equiv 0\bmod p$
if and only if
$H_p(\hat{\ba})\in\Z_p^*$.
\end{proof}

\section{Proof of main theorem}
\label{S:proof}

Let $\cG_1,\ldots,\cG_r$ be fixed subsets in $\Z_{\geq 0}^n$.
Let $\bbA^{\cG_1,\ldots,\cG_r}$ be the space of all algebraic varieties $V(f_1,\ldots,f_r)$
where $f_1,\ldots, f_r$ are supported on $\cG_1,\ldots,\cG_r$, respectively.
Namely, $f_j=\sum_{\bg\in \cG_j} a_{j,\bg} \bx^\bg$ where $a_{j,\bg}\neq 0$ for all $\bg\in\cG_j$.
For any field $K$ containing $\Q$ (resp. $\F_p$) we shall identify 
each algebraic variety $V$ (resp. $\bar{V}$) in $\bbA^{\cG_1,\ldots,\cG_r}(K)$ by
$\ba=(a_{j,\bg})_{j,\bg}$ (resp. $\bar\ba$) with $a_{j,\bg}\in K$ (resp. $\bar{a}_{j,\bg}$) and
$\bg\in\cG_j$ for $1\leq j\leq r$. 
In this section $V_{\ba}$ for the variety defined by $\ba$.

This section is devoted to prove our main Theorem \ref{T:main}.
It is clear that Theorems \ref{T:2} 
and \ref{T:3} below 
proves Parts (1) and (2) of Theorem \ref{T:main}, respectively.

In the following proposition we show that the polynomial 
$H_p^{[a]}(\A)\in\Q[\A]$ defined in (\ref{E:Hpa})
is indeed a Hasse polynomial for the $p$-adic valuation of 
$|V_\ba(\F_q)|$. For any prime $\wp$ over $p$ of residue degree $a$
let $\bar{V_\ba}$ denote $V_\ba\bmod \wp$.

\begin{proposition}\label{P:HassePolynomial}
Let $\bar\Z$ denote the ring of integral elements in $\bar\Q$.
For any $\ba\in\bbA^{\cG_1,\ldots,\cG_r}(\bar\Z)$, or 
for any $\ba$ in $\bbA^{\cG_1,\ldots,\cG_r}(\bar\Q)$ and 
$p$ large enough,
the following statements are equivalent to each other
\begin{enumerate}
\item 
$|H_p^{[a]}(\ba)|_p=1$
\item 
$H_p^{[a]}(\ba)\not\equiv 0\bmod \wp$ 
\item 
$\ord_q |\bar{V_\ba}(\F_q)| = \bmu$.
\end{enumerate}
\end{proposition}
\begin{proof}
For each $\ba$ we may choose for the rest of the proof that 
$p$ is large enough so 
that $\ord_p \ba \geq 0$.
Then let $\hat{\ba}$ be the Teichm\"uller lifting of
$\bar{\ba}$.
The first two statements are clear.
We claim that Parts (2) and (3) are equivalent. 
Applying (\ref{E:dwork2}) and Lemma \ref{L:dominate2}
we have
\begin{eqnarray*}
\frac{|\bar{V_\ba}(\F_q)|}{q^{\bmu}} -q^{n-{\bmu}}
&\equiv &\sum_{(B,C)\in\cK}(q-1)^{|B|+|C|}\frac{\Tr(M_{B,C}(\hat{\ba})^{[a]})}{q^{w_\Z(B,C)}}\bmod \wp
\\\nonumber
&\equiv & \sum_{(B,C)\in\cK}(-1)^{|B|+|C|+aw_\Z(B,C)}\Tr(N_{B,C}(\hat{\ba})^{[a]})\bmod \wp.
\end{eqnarray*}
Comparing to definition (\ref{E:Hpa}) we have
\begin{equation*}
\frac{|\bar{V_\ba}(\F_q)|}{q^{\bmu}}
\equiv H_p^{[a]}(\hat\ba)\equiv H_p^{[a]}(\ba) \bmod \wp.
\end{equation*}
Our claim follows.
\end{proof}

Let $\cZ^{min}$ be as in (\ref{E:Zmin}). 
Let $\cD$ be the set of all denominators $d$ of rational representations of all 
$(\bt,\bv)\in\cZ^{min}$. Let $\theta(\cG_1,\ldots,\cG_r)$ be as in Lemma \ref{L:Hp}.
We shall write $\cP$ for the set of primes $p\equiv 1\bmod d$ for all $d$ lying in $\cD$ 
and $p>\theta(\cG_1,\ldots,\cG_r)$. Observe that it is of positive density in $\Spec(\Z)$
by Dirichlet's theorem on arithmetic progressions.
The following theorem shows that for prime $p$ in $\cP$, 
a subset of $\Spec{\Z}$ of positive density,
there is a {\em Hasse polynomial} $H_p^{[a]}(\A)$
defined over $\Q$ that $H_p^{[a]}(\A)\bmod p$ is a nonzero polynomial
in $\F_p[\A]$ of homogenous degree $(p-1)w_\Z(B,C)$ (as seen 
in Lemma \ref{L:Hp}).

\begin{theorem}\label{T:Hasse}
Let $K$ be a number field and $\cO_K$ its ring of integers.
For any $p \in \cP$ 
let $\cU_p$ be the subset in $\bbA^{\cG_1,\ldots,\cG_r}$
consisting of 
all $\ba$ with $|H_p^{[a]}(\ba)|_p=1$ where $a$ is corresponding residue degree over $p$.
Write $V_\ba$ for the algebraic variety defined by $\ba$ and
$\bar{V_\ba}$ for its reduction $V_\ba\bmod \wp$ with residue degree $a$.
Then for every $V_\ba\in\A^{\cG_1,\ldots,\cG_r}(\cO_K)$ we have 
$V_\ba$ in $\cU_p(\cO_K)$ if and only if $\ord_q|\bar{V_\ba}(\F_q)|=\bmu$.
In particular, 
if $V_\ba\in\cU_p(K)$ and $p$ is large enough we have
$\ord_q|\bar{V_\ba}(\F_q)|=\bmu$.
\end{theorem}

\begin{proof}
It follows immediately 
from Proposition \ref{P:HassePolynomial}.
\end{proof}

\begin{theorem}\label{T:2}
Let $K$ be a number field.
Let $\cU$ be the subset of $\bbA^{\cG_1,\ldots,\cG_r}$ 
consisting of all $\ba$ over $K$ satisfying that $\prod_{p\in\cP} |H_p^{[a(p)]}(\ba)| = 1$ 
where $a(p)$ denote the corresponding residue degree over $p$.
If $V_\ba\in\cU(\cO_K)$ then $\ord_q|\bar{V_\ba}(\F_q)|=\bmu$ for all $p\in\cP$.
If $V_\ba\in\cU(K)$ then $\ord_q|\bar{V_\ba}(\F_q)|=\bmu$ for all 
$p\in\cP$ and $p$ large enough.
\end{theorem}
\begin{proof}
It follows from Theorem \ref{T:Hasse} above.
\end{proof}

In general the above defined Hasse polynomials $H_p^{[a]}(\A)$ in $\Q[\A]$
depends on $p$. However, for special occasions like (but not limited to)
$\cD=\{1\}$ we found that $H_p^{[a]}(\A)\bmod p$ is a constant
independent of $p$. 

\begin{theorem}\label{T:3}
Suppose $\cD=\{1\}$.
Suppose $\bc(\cG_1,\ldots,\cG_r)\neq 0$ where $\bc(\cdot)$ is defined in (\ref{E:conditional}).
Then for all $V$ in $\bbA^{\cG_1,\ldots,\cG_r}(\Q)$ and $p$ large enough
we have 
$\ord_p |\bar{V}(\F_p)| = \bmu$.
\end{theorem}

\begin{proof}
By Proposition \ref{P:conditional}  we know that 
for all $\ba$ in $\bbA^{\cG_1,\ldots,\cG_r}(\Q)$
for $p$ large enough
$H_p(\ba)\in\Z_p^*\cap\Q$.
Thus by Proposition \ref{P:HassePolynomial} we have
$\ord_p|\bar{V}(\F_p)| = \bmu$ 
\end{proof}

We remark that the condition $\cD=\{1\}$ is an indicator of how sparse the
given system $\cG_1,\ldots,\cG_r$ is. We have the following more explicit 
criterion in constructing such supporting sets. 

\begin{proposition}\label{P:D}
Let $\cG_1,\ldots,\cG_r$ be given subsets of points in $\Z^n_{\geq 0}$.
If
$\sum_{\bg\in\cG_{j,C}}u_\bg \bg \in\Z_{\ge 1}^C$ for each $1\leq j\leq r$
with $\sum_{\bg\cG_{j,C}} u_\bg \in\Z_{\geq 1}$ and $u_\bg\in\Q_{\geq 0}$
implies that $u_\bg\in \Z$ then 
$\cD=\{1\}$.
\end{proposition}

\begin{remark}
Applying Proposition \ref{P:D}, we have $\cD=\{1\}$ for 
Example \ref{Ex:2}. 
\end{remark}

\end{document}